\documentclass[10pt]{amsart}

\usepackage{amsmath,amssymb,amsthm,mathtools}
\usepackage[margin=1.1in]{geometry}
\usepackage[colorlinks=true,linkcolor=blue,citecolor=blue,urlcolor=blue]{hyperref}
\usepackage{array}
\usepackage{booktabs}

\newtheorem{theorem}{Theorem}[section]
\newtheorem{lemma}[theorem]{Lemma}

\theoremstyle{remark}

\newcommand{\ph}{\varphi}
\newcommand{\dd}{\delta}

\title{A Counterexample to the Chung--Graham--Spiro Gap-Set Conjecture}

\author{Mohsen Aliabadi}
\address{Department of Mathematics, Clayton State University, Morrow, GA, USA}
\email{mohsen.aliabadi@clayton.edu}



\begin{document}

\begin{abstract}
Chung, Graham, and Spiro introduced slow Fibonacci walks and used them to
partition the integers $n\ge2$ into two sequences, the down-integers and the
up-integers. They studied the local spacing of these two sequences and
conjectured that their $\ell$-step gap sets agree for every $\ell\ge1$. We show
that the conjecture fails at $\ell=4$ by proving
\[
        9\in U_4\setminus D_4 .
\]
\end{abstract}

\maketitle

\section{Introduction}

Let
\[
        f_1=f_2=1,\qquad f_{k+2}=f_{k+1}+f_k\quad(k\ge1)
\]
be the Fibonacci sequence, and let
\[
        \ph=\frac{1+\sqrt5}{2}.
\]
Given positive integers $a_1,a_2$, the corresponding \emph{Fibonacci walk} is the
sequence
\[
        w_1=a_1,\qquad w_2=a_2,\qquad w_{k+2}=w_{k+1}+w_k .
\]
If $w_s=n$ for some $s$, then this walk reaches $n$. Among all Fibonacci walks
that reach a fixed integer $n$, Chung, Graham, and Spiro \cite{CGS} studied
those that reach $n$ as late as possible; these are the slow Fibonacci walks.

We shall use the following representation theorem from \cite{CGS}. For every
integer $n\ge2$, there exist unique integers $a,b,t$ such that
\[
        n=a f_t+b f_{t-1},
        \qquad
        t\ge2,
        \qquad
        1\le a\le b\le f_t.
\]
We call this the \emph{Chung--Graham--Spiro representation} of $n$, and we call
$t$ the \emph{representation parameter} of $n$. A representation satisfying the
displayed conditions will be called \emph{valid}.

The slow walk reaching $n$ has a next term. In the terminology of \cite{CGS}, an
integer $n\ge2$ is called a \emph{down-integer} if this next term is
$\lfloor\ph n\rfloor$, and it is called an \emph{up-integer} if this next term is
$\lceil\ph n\rceil$. The Chung--Graham--Spiro characterization says that $n$ is
a down-integer precisely when its representation parameter $t$ is even, and that
$n$ is an up-integer precisely when $t$ is odd. Thus the down-integers and the
up-integers form a disjoint partition of the integers $n\ge2$.

Let
\[
        D=\{d_1<d_2<d_3<\cdots\}
\]
be the increasing sequence of down-integers, and let
\[
        U=\{u_1<u_2<u_3<\cdots\}
\]
be the increasing sequence of up-integers. The first few values are
\[
\begin{aligned}
        D&=\{2,5,7,9,10,12,13,15,18,\ldots\},\\
        U&=\{3,4,6,8,11,14,16,17,\ldots\}.
\end{aligned}
\]
These lists suggest that both sequences have very small local gaps between consecutive terms . Indeed, it
was proved in \cite{CGS} that the possible gaps between consecutive terms are
the same for the two sequences, namely
\[
        D_1=U_1=\{1,2,3,5\}.
\]
Thus every two consecutive down-integers, and every two consecutive
up-integers, differ by at most $5$.

It is then natural to ask whether the same agreement persists for longer local
spacings. For $\ell\ge1$, define the \emph{$\ell$-step gap sets}
\[
        D_\ell=\{d_{k+\ell}-d_k:k\ge1\},
        \qquad
        U_\ell=\{u_{k+\ell}-u_k:k\ge1\}.
\]
Thus $D_\ell$ records the possible spans of blocks of $\ell+1$ consecutive
down-integers, and $U_\ell$ has the analogous meaning for up-integers. Chung,
Graham, and Spiro also proved that
\[
        D_2=U_2=\{2,3,4,5,6,8,10\},
\]
and conjectured that
\[
        D_\ell=U_\ell\qquad\text{for every }\ell\ge1.
\]
The purpose of this note is to give a counterexample.

\begin{theorem}\label{thm:main}
One has
\[
        9\in U_4\setminus D_4.
\]
In particular,
\[
        D_4\ne U_4.
\]
\end{theorem}

We first note that $9\in U_4$. Using the representation criterion, the integers
$2\le m\le17$ split as
\[
\begin{aligned}
        D\cap[2,17]&=\{2,5,7,9,10,12,13,15\},\\
        U\cap[2,17]&=\{3,4,6,8,11,14,16,17\}.
\end{aligned}
\]
Hence
\[
        8,11,14,16,17
\]
are five consecutive up-integers, and
\[
        17-8=9.
\]
Therefore $9\in U_4$. It remains to prove that $9\notin D_4$.

For a real number $x$, let
\[
        \{x\}=x-\lfloor x\rfloor
\]
denote its fractional part. For an integer $n\ge1$, set
\[
        \dd_n=\{\ph n\}.
\]
Also define
\[
        \eta=\frac{1}{\sqrt5\,\ph}=\frac{1}{\ph+2},
        \qquad
        \rho=1-\eta.
\]

We shall use the following consequence of Proposition 4.2 of \cite{CGS}.

\begin{lemma}\label{lem:threshold}
For every integer $m\ge2$,
\[
        \dd_m<\eta \implies m\in D.
\]
Consequently,
\[
        m\in U \implies \dd_m\ge \eta,
        \qquad
        m\in D \implies \dd_m\le \rho.
\]
\end{lemma}

\begin{proof}
The first implication is Proposition 4.2 of \cite{CGS}. Since each integer
$m\ge2$ lies in exactly one of $D$ and $U$, its contrapositive gives
\[
        m\in U\implies \dd_m\ge\eta.
\]
The same proposition also gives the upper-end implication
\[
        1-\dd_m<\eta \implies m\in U.
\]
Therefore, if $m\in D$, then $1-\dd_m\ge\eta$, and hence
\[
        \dd_m\le1-\eta=\rho.
\]
\end{proof}

For $0\le j\le9$, put
\[
        \alpha_j=\{j\ph\}.
\]
Then, for every integer $n\ge1$,
\[
        \dd_{n+j}=\{\dd_n+\alpha_j\}.
\]
Explicitly,
\[
\begin{aligned}
\alpha_0&=0,\\
\alpha_1&=\ph-1,\\
\alpha_2&=2\ph-3,\\
\alpha_3&=3\ph-4,\\
\alpha_4&=4\ph-6,\\
\alpha_5&=5\ph-8,\\
\alpha_6&=6\ph-9,\\
\alpha_7&=7\ph-11,\\
\alpha_8&=8\ph-12,\\
\alpha_9&=9\ph-14.
\end{aligned}
\]

We shall also use two elementary Fibonacci identities. For even $t$,
\[
        8=f_{t-7}f_t-f_{t-6}f_{t-1},
\]
and
\[
        5=f_{t-5}f_{t-1}-f_{t-6}f_t.
\]
These follow from d'Ocagne's identity
\[
        f_r f_{s+1}-f_{r+1}f_s=(-1)^s f_{r-s},
\]
with the standard extension
\[
        f_{-m}=(-1)^{m+1}f_m.
\]
For background on Fibonacci identities, see \cite{Koshy}.

\section{Two shift lemmas}

The following two lemmas are the main structural ingredients in the proof.

\begin{lemma}[Forward $8$-shift]\label{lem:forward8}
Let $N\in D$, and write its Chung--Graham--Spiro representation as
\[
        N=a f_t+b f_{t-1},
        \qquad
        1\le a\le b\le f_t.
\]
Assume that $t$ is even, $t\ge8$, and
\[
        \dd_N\ge 15-9\ph.
\]
Then
\[
        N+8\in D.
\]
\end{lemma}

\begin{proof}
Since $N\in D$ and $t$ is even, Lemma 4.1 of \cite{CGS} gives
\[
        \dd_N=\ph^{-t}(\ph b-a).
\]
Thus
\[
        \ph b-a\ge (15-9\ph)\ph^t.
\]
We first prove that
\[
        b-a\ge f_{t-5}.
\]
Suppose not. Since $b-a$ is an integer,
\[
        b-a\le f_{t-5}-1.
\]
Using $b\le f_t$, we get
\[
        \ph b-a=\ph^{-1}b+(b-a)
        \le \ph^{-1}f_t+f_{t-5}-1.
\]

If $t=8$, then
\[
        \ph^{-1}f_8+f_3-1=21\ph^{-1}+1=21\ph-20,
\]
whereas
\[
        (15-9\ph)\ph^8=9\ph+6.
\]
The inequality
\[
        21\ph-20<9\ph+6
\]
is equivalent to $12\ph<26$, which is true. Hence the desired contradiction
holds when $t=8$.

Now assume $t\ge10$. We use the elementary estimate
\[
        f_s\le \frac{23}{50}\ph^s\qquad(s\ge4),
\]
which follows directly from Binet's formula \cite{Koshy}. Since $t\ge10$, both
$t\ge4$ and $t-5\ge5$, so
\[
        f_t\le \frac{23}{50}\ph^t,
        \qquad
        f_{t-5}\le \frac{23}{50}\ph^{t-5}.
\]
Therefore
\[
        \ph^{-1}f_t+f_{t-5}-1
        <
        \frac{23}{50}(\ph^{-1}+\ph^{-5})\ph^t.
\]
Using $\ph^2=\ph+1$, one obtains
\[
        15-9\ph-\frac{23}{50}(\ph^{-1}+\ph^{-5})
        =
        \frac{957-588\ph}{50}>0.
\]
Hence
\[
        \ph^{-1}f_t+f_{t-5}-1
        <
        (15-9\ph)\ph^t,
\]
contradicting the lower bound for $\ph b-a$. Therefore
\[
        b-a\ge f_{t-5}.
\]

For even $t$,
\[
        8=f_{t-7}f_t-f_{t-6}f_{t-1}.
\]
Thus
\[
\begin{aligned}
        N+8
        &=a f_t+b f_{t-1}+8\\
        &=(a+f_{t-7})f_t+(b-f_{t-6})f_{t-1}.
\end{aligned}
\]
Set
\[
        a'=a+f_{t-7},
        \qquad
        b'=b-f_{t-6}.
\]
Since
\[
        b-a\ge f_{t-5}=f_{t-6}+f_{t-7},
\]
we have
\[
        b'-a'=b-a-f_{t-6}-f_{t-7}\ge0.
\]
Also $a'\ge1$ and $b'\le b\le f_t$. Hence
\[
        1\le a'\le b'\le f_t.
\]
Therefore $N+8$ has a valid Chung--Graham--Spiro representation with the same
even representation parameter $t$. Hence $N+8\in D$.
\end{proof}

\begin{lemma}[Backward $5$-shift]\label{lem:backward5}
Let $N\in D$, and write its Chung--Graham--Spiro representation as
\[
        N=a f_t+b f_{t-1},
        \qquad
        1\le a\le b\le f_t.
\]
Assume that $t$ is even, $t\ge8$, and
\[
        \dd_N\ge 9\ph-14.
\]
Then
\[
        N-5\in D.
\]
\end{lemma}

\begin{proof}
Since $N\in D$ and $t$ is even, Lemma 4.1 of \cite{CGS} gives
\[
        \dd_N=\ph^{-t}(\ph b-a).
\]
Thus
\[
        \ph b-a\ge (9\ph-14)\ph^t.
\]
We first prove that
\[
        b-a\ge f_{t-4}.
\]
Suppose not. Since $b-a$ is an integer,
\[
        b-a\le f_{t-4}-1.
\]
Using $b\le f_t$, we obtain
\[
        \ph b-a=\ph^{-1}b+(b-a)
        \le \ph^{-1}f_t+f_{t-4}-1.
\]
Since $t\ge8$, both $t\ge4$ and $t-4\ge4$. Hence
\[
        f_t\le \frac{23}{50}\ph^t,
        \qquad
        f_{t-4}\le \frac{23}{50}\ph^{t-4}.
\]
Therefore
\[
        \ph^{-1}f_t+f_{t-4}-1
        <
        \frac{23}{50}(\ph^{-1}+\ph^{-4})\ph^t.
\]
Using $\ph^2=\ph+1$, one obtains
\[
        9\ph-14-\frac{23}{50}(\ph^{-1}+\ph^{-4})
        =
        \frac{248\ph-396}{25}>0.
\]
Thus
\[
        \ph^{-1}f_t+f_{t-4}-1
        <
        (9\ph-14)\ph^t,
\]
contradicting the lower bound for $\ph b-a$. Therefore
\[
        b-a\ge f_{t-4}.
\]

For even $t$,
\[
        5=f_{t-5}f_{t-1}-f_{t-6}f_t.
\]
Hence
\[
\begin{aligned}
        N-5
        &=a f_t+b f_{t-1}-5\\
        &=(a+f_{t-6})f_t+(b-f_{t-5})f_{t-1}.
\end{aligned}
\]
Set
\[
        a'=a+f_{t-6},
        \qquad
        b'=b-f_{t-5}.
\]
Since
\[
        b-a\ge f_{t-4}=f_{t-5}+f_{t-6},
\]
we have
\[
        b'-a'=b-a-f_{t-5}-f_{t-6}\ge0.
\]
Also $a'\ge1$ and $b'\le b\le f_t$. Hence
\[
        1\le a'\le b'\le f_t.
\]
Therefore $N-5$ has a valid Chung--Graham--Spiro representation with the same
even representation parameter $t$. Hence $N-5\in D$.
\end{proof}

\section{Proof of the counterexample}

We first dispose of a finite initial range.

\begin{lemma}\label{lem:finite}
There are no five consecutive down-integers whose first-to-fifth difference is
$9$ and such that at least one of the five has representation parameter
$t<8$.
\end{lemma}

\begin{proof}
By five consecutive down-integers, we mean five terms
\[
        d_k,d_{k+1},d_{k+2},d_{k+3},d_{k+4}
\]
of the increasing sequence $D$.

If $M\in D$ has even representation parameter $t<8$, then
\[
        t\in\{2,4,6\}.
\]
Writing
\[
        M=a f_t+b f_{t-1},
        \qquad
        1\le a\le b\le f_t,
\]
we get
\[
        M\le f_t^2+f_t f_{t-1}=f_t f_{t+1}\le f_6f_7=104.
\]
Thus any interval $[n,n+9]$ containing such an $M$ is contained in $[2,113]$.

The finite enumeration is reproduced in the appendix. It implements the
Chung--Graham--Spiro representation directly and gives
\[
\begin{aligned}
D\cap[2,113]=\{&
2,5,7,9,10,12,13,15,18,23,26,28,31,33,34,36,38,39,41,43,\\
&44,46,47,48,49,51,52,54,56,57,59,60,62,64,65,67,68,70,\\
&72,73,75,78,80,81,83,86,88,89,91,94,96,99,102,104,107,112\}.
\end{aligned}
\]
Checking the consecutive five-term blocks in this displayed list shows that no
five consecutive elements of $D\cap[2,113]$ have first-to-fifth difference $9$.
\end{proof}

We now prove that $9\notin D_4$. Suppose, for contradiction, that $9\in D_4$.
Then there exist five consecutive down-integers
\[
        n=d_k,\quad d_{k+1},\quad d_{k+2},\quad d_{k+3},\quad d_{k+4}=n+9.
\]
Let
\[
        g_i=d_{k+i}-d_{k+i-1}\qquad(1\le i\le4).
\]
Since $D_1=\{1,2,3,5\}$,
\[
        g_i\in\{1,2,3,5\}.
\]
Since $D_2=\{2,3,4,5,6,8,10\}$,
\[
        g_i+g_{i+1}\in\{2,3,4,5,6,8,10\}
        \qquad(1\le i\le3).
\]
Finally,
\[
        g_1+g_2+g_3+g_4=9.
\]
A direct enumeration gives exactly the following twenty-two possible patterns:
\[
\begin{aligned}
&(1,2,1,5),(1,2,3,3),(1,3,2,3),(1,3,3,2),(1,5,1,2),\\
&(2,1,1,5),(2,1,3,3),(2,1,5,1),(2,2,2,3),(2,2,3,2),\\
&(2,3,1,3),(2,3,2,2),(2,3,3,1),\\
&(3,1,2,3),(3,1,3,2),(3,2,1,3),(3,2,2,2),(3,2,3,1),\\
&(3,3,1,2),(3,3,2,1),(5,1,1,2),(5,1,2,1).
\end{aligned}
\]

For a pattern $P=(g_1,g_2,g_3,g_4)$, define its \emph{position set} $S(P)$ by
\[
        S(P)=\{0,g_1,g_1+g_2,g_1+g_2+g_3,9\}.
\]
Thus $S(P)$ is the set of positions in the interval $[n,n+9]$ occupied by the
five down-integers. Put
\[
        x=\dd_n.
\]
If $j\in S(P)$, then $n+j\in D$, so Lemma~\ref{lem:threshold} gives
\[
        \{x+\alpha_j\}\le \rho.
\]
If $j\notin S(P)$, then $n+j\notin D$. Since $D$ and $U$ partition the integers
$m\ge2$, this means $n+j\in U$, and Lemma~\ref{lem:threshold} gives
\[
        \{x+\alpha_j\}\ge \eta.
\]

The exact interval computations in Section~\ref{sec:intervals} show that twelve
of the twenty-two patterns have no feasible value of $x$, namely
\[
\begin{aligned}
&(1,5,1,2),(2,1,1,5),(2,1,3,3),(2,1,5,1),(2,2,2,3),(2,2,3,2),\\
&(3,2,1,3),(3,2,2,2),(3,3,1,2),(3,3,2,1),(5,1,1,2),(5,1,2,1).
\end{aligned}
\]
For the remaining ten patterns, the same exact computation gives the feasible
intervals
\[
\begin{array}{c|c}
P & \text{feasible interval for }x=\dd_n\\
\hline
(1,2,1,5) & [\,\eta+1-\alpha_6,\ 1-\alpha_7\,]\\
(1,2,3,3) & [\,1-\alpha_9,\ 1-\alpha_4\,]\\
(1,3,2,3) & [\,1-\alpha_4,\ 1-\alpha_7\,]\\
(1,3,3,2) & [\,1-\alpha_7,\ \rho\,]\\
(2,3,1,3) & [\,0,\ \rho-\alpha_6\,]\\
(2,3,2,2) & [\,0,\ 1-\alpha_8\,]\\
(2,3,3,1) & [\,1-\alpha_8,\ 1-\alpha_3\,]\\
(3,1,2,3) & [\,\eta+1-\alpha_1,\ 1-\alpha_7\,]\\
(3,1,3,2) & [\,1-\alpha_7,\ \rho\,]\\
(3,2,3,1) & [\,1-\alpha_3,\ \rho-\alpha_9\,].
\end{array}
\]

Consider first the six feasible patterns
\[
\begin{aligned}
&(1,2,1,5),\quad (1,2,3,3),\quad (1,3,2,3),\quad (1,3,3,2),\\
&(3,1,2,3),\quad (3,1,3,2).
\end{aligned}
\]
For each of these patterns, the displayed feasible interval implies
\[
        \dd_n\ge 15-9\ph.
\]
Moreover, each of these six patterns has
\[
        n+8\in U.
\]
By Lemma~\ref{lem:finite}, the down-integer $n$ has even representation
parameter $t\ge8$. Hence Lemma~\ref{lem:forward8} gives
\[
        n+8\in D,
\]
contradicting $n+8\in U$.

It remains to consider the four feasible patterns
\[
        (2,3,1,3),\quad (2,3,2,2),\quad (2,3,3,1),\quad (3,2,3,1).
\]
Set
\[
        N=n+9.
\]
Then $N\in D$. For each of these four patterns, the displayed feasible interval
implies
\[
        x\le \rho-\alpha_9.
\]
Thus
\[
        x+\alpha_9\le \rho<1,
\]
and so
\[
        \dd_N=\dd_{n+9}=\{x+\alpha_9\}=x+\alpha_9.
\]
Since $x\ge0$, this gives
\[
        \dd_N\ge \alpha_9=9\ph-14.
\]
Each of these four patterns also has
\[
        n+4=N-5\in U.
\]
By Lemma~\ref{lem:finite}, the down-integer $N$ has even representation
parameter $t\ge8$. Hence Lemma~\ref{lem:backward5} gives
\[
        N-5=n+4\in D,
\]
contradicting $n+4\in U$.

Thus every possible pattern is impossible. Therefore
\[
        9\notin D_4.
\]
Since $9\in U_4$, we conclude that
\[
        9\in U_4\setminus D_4.
\]
This proves Theorem~\ref{thm:main}.

\section{Exact interval certificates}\label{sec:intervals}

This section records the exact interval computations used in the proof. All
computations take place in the quadratic field $\mathbb Q(\ph)$ and use only the
relation
\[
        \ph^2=\ph+1.
\]

For $0\le j\le9$, define
\[
        I_D(j)=\{x\in[0,1):\{x+\alpha_j\}\le\rho\},
\]
and
\[
        I_U(j)=\{x\in[0,1):\{x+\alpha_j\}\ge\eta\}.
\]
For a pattern $P$, the feasible set for $x=\dd_n$ is obtained by intersecting
$I_D(j)$ over all $j\in S(P)$ and $I_U(j)$ over all $j\notin S(P)$.

For $0\le\alpha<1$, the condition
\[
        \{x+\alpha\}\le\rho
\]
is equivalent to
\[
        x\in
        \bigl([0,\rho-\alpha]\cap[0,1)\bigr)
        \cup
        \bigl([1-\alpha,1+\rho-\alpha]\cap[0,1)\bigr).
\]
Similarly,
\[
        \{x+\alpha\}\ge\eta
\]
is equivalent to
\[
        x\in
        \bigl([\eta-\alpha,1-\alpha]\cap[0,1)\bigr)
        \cup
        \bigl([1+\eta-\alpha,1]\cap[0,1)\bigr).
\]
Here endpoints equal to $1$ are harmless, since throughout $x\in[0,1)$.

For each pattern $P$, the intervals below are obtained by intersecting the five
$I_D(j)$-conditions for $j\in S(P)$ and the five $I_U(j)$-conditions for
$j\notin S(P)$. Substituting $\alpha=\alpha_j$, using
\[
        \eta=\frac{1}{\ph+2},
        \qquad
        \rho=1-\eta,
\]
and reducing by $\ph^2=\ph+1$ gives the following certificates.

The next table proves that twelve patterns have empty feasible set. In each row,
the displayed interval is forced by some of the necessary conditions for the
pattern, while the final listed necessary condition is disjoint from that
interval. Hence the full feasible intersection is empty.

\[
\begin{array}{c|c|c}
P & \text{forced interval} & \text{additional necessary condition}\\
\hline
(1,5,1,2)
&
\left[\frac{28}{5}-\frac{16}{5}\ph,\ 7-4\ph\right]
&
I_D(7)=\left[0,\frac{57}{5}-\frac{34}{5}\ph\right]\cup[12-7\ph,1]
\\[2mm]

(2,1,1,5)
&
\left[\frac{43}{5}-\frac{26}{5}\ph,\ 2-\ph\right]
&
I_D(9)=\left[0,\frac{72}{5}-\frac{44}{5}\ph\right]\cup[15-9\ph,1]
\\[2mm]

(2,1,3,3)
&
\left[\frac{43}{5}-\frac{26}{5}\ph,\ 2-\ph\right]
&
I_D(9)=\left[0,\frac{72}{5}-\frac{44}{5}\ph\right]\cup[15-9\ph,1]
\\[2mm]

(2,1,5,1)
&
\left[\frac{43}{5}-\frac{26}{5}\ph,\ 2-\ph\right]
&
I_D(9)=\left[0,\frac{72}{5}-\frac{44}{5}\ph\right]\cup[15-9\ph,1]
\\[2mm]

(2,2,2,3)
&
[0,\ 5-3\ph]
&
I_U(5)=\left[\frac{43}{5}-\frac{26}{5}\ph,\ 9-5\ph\right]
\\[2mm]

(2,2,3,2)
&
[0,\ 5-3\ph]
&
I_U(5)=\left[\frac{43}{5}-\frac{26}{5}\ph,\ 9-5\ph\right]
\\[2mm]

(3,2,1,3)
&
[10-6\ph,\ 2-\ph]
&
I_D(9)=\left[0,\frac{72}{5}-\frac{44}{5}\ph\right]\cup[15-9\ph,1]
\\[2mm]

(3,2,2,2)
&
[5-3\ph,\ 10-6\ph]
&
I_U(8)=\left[0,13-8\ph\right]\cup
\left[\frac{68}{5}-\frac{41}{5}\ph,1\right]
\\[2mm]

(3,3,1,2)
&
\left[\frac{43}{5}-\frac{26}{5}\ph,\ 2-\ph\right]
&
I_D(9)=\left[0,\frac{72}{5}-\frac{44}{5}\ph\right]\cup[15-9\ph,1]
\\[2mm]

(3,3,2,1)
&
\left[\frac{43}{5}-\frac{26}{5}\ph,\ 2-\ph\right]
&
I_D(9)=\left[0,\frac{72}{5}-\frac{44}{5}\ph\right]\cup[15-9\ph,1]
\\[2mm]

(5,1,1,2)
&
\left[\frac{18}{5}-\frac{11}{5}\ph,\ 5-3\ph\right]
&
I_D(6)=\left[0,\frac{47}{5}-\frac{29}{5}\ph\right]\cup[10-6\ph,1]
\\[2mm]

(5,1,2,1)
&
\left[\frac{18}{5}-\frac{11}{5}\ph,\ 5-3\ph\right]
&
I_D(6)=\left[0,\frac{47}{5}-\frac{29}{5}\ph\right]\cup[10-6\ph,1].
\end{array}
\]

For example, in the first row the forced interval is
\[
        \left[\frac{28}{5}-\frac{16}{5}\ph,\ 7-4\ph\right].
\]
The additional condition $I_D(7)$ requires
\[
        x\le \frac{57}{5}-\frac{34}{5}\ph
        \qquad\text{or}\qquad
        x\ge 12-7\ph.
\]
But
\[
        \frac{57}{5}-\frac{34}{5}\ph
        <
        \frac{28}{5}-\frac{16}{5}\ph
        \le x
        \le
        7-4\ph
        <
        12-7\ph.
\]
Thus the intersection is empty. The other rows are verified in exactly the same
way.

For the ten nonempty cases, the same intersection procedure gives the following
exact feasible intervals:
\[
\begin{array}{c|c}
P & \text{feasible interval}\\
\hline
(1,2,1,5) & [\,\eta+1-\alpha_6,\ 1-\alpha_7\,]\\
(1,2,3,3) & [\,1-\alpha_9,\ 1-\alpha_4\,]\\
(1,3,2,3) & [\,1-\alpha_4,\ 1-\alpha_7\,]\\
(1,3,3,2) & [\,1-\alpha_7,\ \rho\,]\\
(2,3,1,3) & [\,0,\ \rho-\alpha_6\,]\\
(2,3,2,2) & [\,0,\ 1-\alpha_8\,]\\
(2,3,3,1) & [\,1-\alpha_8,\ 1-\alpha_3\,]\\
(3,1,2,3) & [\,\eta+1-\alpha_1,\ 1-\alpha_7\,]\\
(3,1,3,2) & [\,1-\alpha_7,\ \rho\,]\\
(3,2,3,1) & [\,1-\alpha_3,\ \rho-\alpha_9\,].
\end{array}
\]
These are precisely the intervals used in the proof of Theorem~\ref{thm:main}.
\section{Discussion and open questions}

Note that the case $\ell=3$ is not settled by the present argument. Our numerical computations have not revealed any discrepancy between $D_3$ and $U_3$, providing some computational evidence that

$$
D_3=U_3.
$$

It would be interesting to determine whether this equality indeed holds. More generally, for which values of $\ell$ do we have

$$
D_\ell=U_\ell?
$$

In particular, does the failure exhibited here at $\ell=4$ persist for all, or infinitely many, larger values of $\ell$?
\appendix

\section*{Appendix: finite verification}

The following Sage code reproduces the finite enumeration used in
Lemma~\ref{lem:finite}. It implements the Chung--Graham--Spiro representation
directly and checks that no five consecutive down-integers in the required
finite range have span $9$.

\begin{verbatim}
def fibs_up_to(N):
    f = [0, 1, 1]
    while f[-1] <= N:
        f.append(f[-1] + f[-2])
    return f

def cgs_representation(m):
    f = fibs_up_to(m)
    for t in range(2, len(f)):
        if f[t] + f[t-1] > m:
            break
        for a in range(1, f[t] + 1):
            for b in range(a, f[t] + 1):
                if a*f[t] + b*f[t-1] == m:
                    return (a, b, t)
    raise ValueError("No representation found")

D = []

for m in range(2, 114):
    a, b, t = cgs_representation(m)
    if t % 2 == 0:
        D.append(m)

print("D cap [2,113] =")
print(D)

bad_blocks = []
for i in range(len(D) - 4):
    block = D[i:i+5]
    if block[-1] - block[0] == 9:
        bad_blocks.append(block)

print("Five-term down-blocks with span 9:")
print(bad_blocks)
\end{verbatim}

The output is the displayed list of $D\cap[2,113]$ in Lemma~\ref{lem:finite},
followed by the empty list of five-term down-blocks with span $9$.

\medskip

\noindent\textbf{Conflict of interest.}
The author declares that there is no conflict of interest.

\medskip

\noindent\textbf{Data availability.}
No datasets were generated or analyzed during the present work.

\end{document}